\documentclass[11pt,a4paper]{amsproc}
\usepackage{amsthm,wrapfig,amsmath,amssymb,amsfonts,setspace,verbatim,graphicx,mathtools,mathrsfs,commath,float,mdframed,frame,xcolor,afterpage,enumitem,ulem}
\usepackage[hidelinks]{hyperref}
\usepackage[T1]{fontenc}

\newtheorem{ut}{Theorem}
\newtheorem{up}[ut]{Proposition}
\newtheorem{ul}[ut]{Lemma}

\newtheorem{uobs}[ut]{Observation}

\theoremstyle{remark}

\theoremstyle{definition}
\newtheorem{ur}{Remark}

\begin{document}

\author[D.S. Lipham]{David Sumner Lipham}



\title[On the closure of a  plane ray]{On the closure of a  plane ray \\ that limits onto itself}

\begin{abstract}We show that the closure of any self-entwined ray in the plane must contain a Cantor set of mutually disjoint continua. This is  false in dimension three. 
\end{abstract}

\maketitle

\

\

\section{Introduction}

In this paper we study a planar form of entanglement or recurrence which gives rise to a complicated  structure within the surrounding space. Specifically we examine properties of planar continua which contain self-entwined rays (i.e.\ rays that  \textit{limit onto themselves}). We show that these continua have  uncountable structures. 

\subsection{Terminology}A \textbf{ray} is a continuous one-to-one image of  $[0,\infty)$. A ray  is \textbf{self-entwined} if no open set intersects it in an arc. This  means  that each point   is revisited to within $\varepsilon$ by points further along the ray (see Section 2.2). A \textbf{continuum} is a compact connected metric space. A continuum is \textbf{indecomposable} if it cannot be written as a union of two proper subcontinua. A continuum  is \textbf{non-Suslinian} if it contains an uncountable collection of pairwise-disjoint, non-degenerate subcontinua. 

Indecomposable continua are known to be non-Suslinian (cf.\ \cite[\S1.1]{lel}). 

\subsection{Motivation and main result} 
Self-entwined rays are commonly found in continua which arise from folding algorithms and chaotic systems. For example, they exist within the arc components of Plykin and Ikeda attractors (see Figures 1 and 2). These planar attractors  are non-Suslinian because they are indecomposable; they   even contain  neighborhoods that are  homeomorphic to  the Cantor fence $2^\omega\times [0,1]$.  

The goal of this paper is to  establish the general result:

\begin{ut}\label{t1}If $X$ is a self-entwined ray in $\mathbb R^2$, then $\overline X$ is non-Suslinian.\end{ut}
 
Although every ray is a \textbf{rational} space, meaning that it has a basis of open sets with countable boundaries, Theorem \ref{t1} implies that the closure of a self-entwined plane ray cannot be rational (cf.\ \cite[\S1.3]{lel}).

\subsection{Context and limitations}
Jones \cite{jon1,jon2} proved that a plane ray cannot be locally connected unless it is locally compact. From \cite{lm} it follows that each locally connected plane ray is homeomorphic to either $[0,\infty)$, the unit circle $S^1$, or $S^1\cup [0,1]$. In particular, planar self-entwined rays are not  locally connected.

  Tymchatyn \cite[Example 3]{tym} constructed a rational  continuum which is the closure of a self-entwined ray.    So  Theorem \ref{t1} is false outside the plane. Tymchatyn's continuum, which can be embedded into $\mathbb R^3$,  was already known to be non-planar because it is hereditarily locally connected but not finitely-Suslinian \cite{lel}. Moreover, the ray in this continuum is locally connected.

  Curry \cite{cur} showed that  the closure of a self-entwined plane ray is an indecomposable continuum provided that it is $1$-dimensional and non-separating or finitely separating. On the other hand, the Sierpiński carpet, which is $1$-dimensional and locally connected with infinitely-many complementary components, is easily seen to contain a dense self-entwined ray. So the non-Suslinian property cannot always be derived from Curry's result.

\begin{figure}[h]
\includegraphics[scale=0.058]{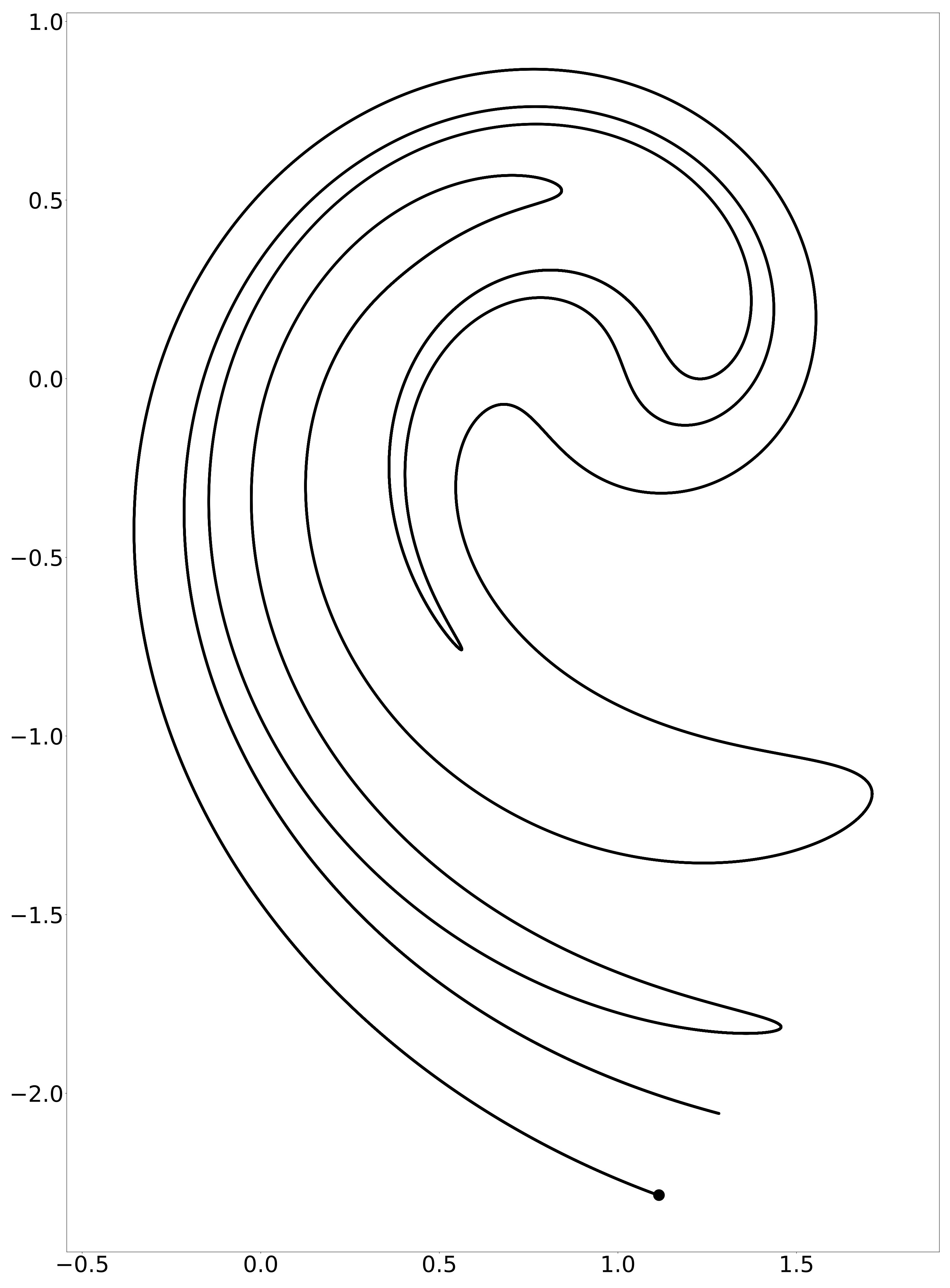}\hspace{0.1in}
\includegraphics[scale=0.058]{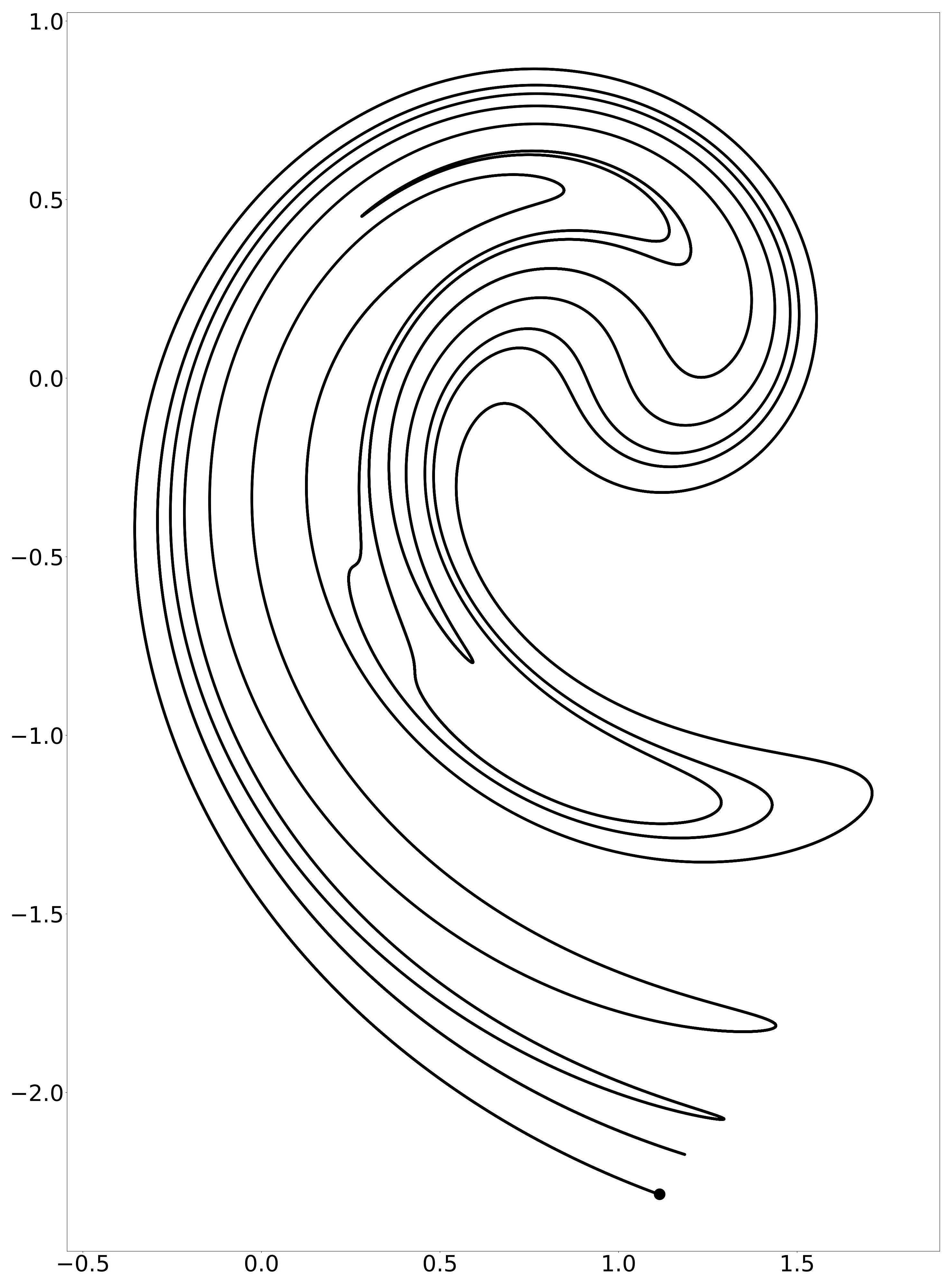}
\caption{The Ikeda map $$\textstyle f(z)=1+0.9z\exp\big(i(0.4-\frac{6}{1+|z|^2})\big)\hspace{2cm}$$  produces a strange attractor $K$ which contains   a fixed point  $p\approx  1.114- 2.285i$. The arc component of $p$ in $K$ is a self-entwined ray with initial point $p$. Iterates of  $$I=\{x-2.285i:0\leq x\leq 1.114\}\hspace{0.8in}$$ approximate initial segments of the ray since $f$ is injective on $I$, and  $I\setminus \{p\}$ is contained in the basin of attraction for $K$. The arcs $f^6[I]$ and $f^7[I]$ are shown above. 
}
\end{figure}

\begin{figure}[h]
\includegraphics[scale=0.065]{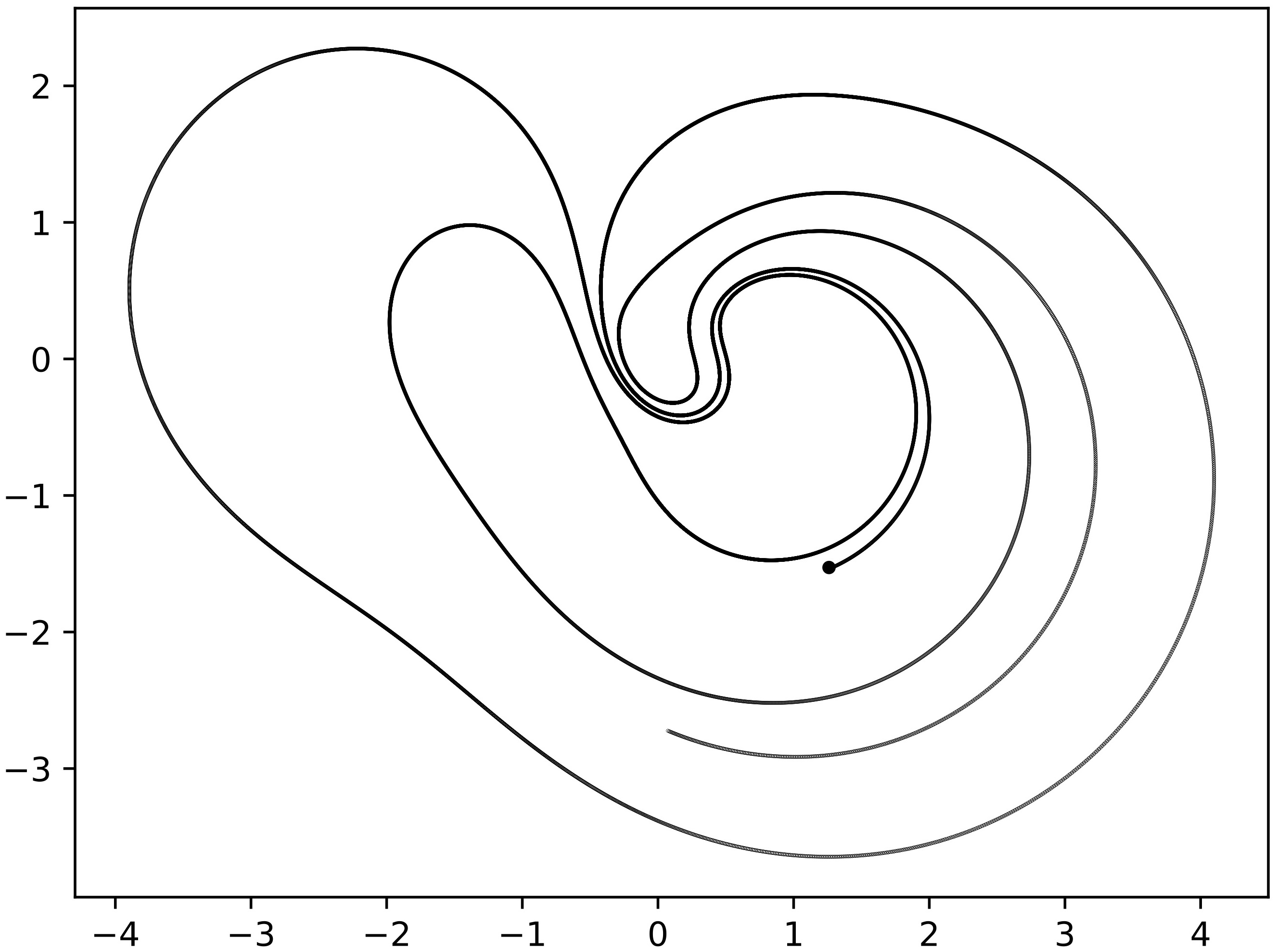}\hspace{0.1in}
\includegraphics[scale=0.065]{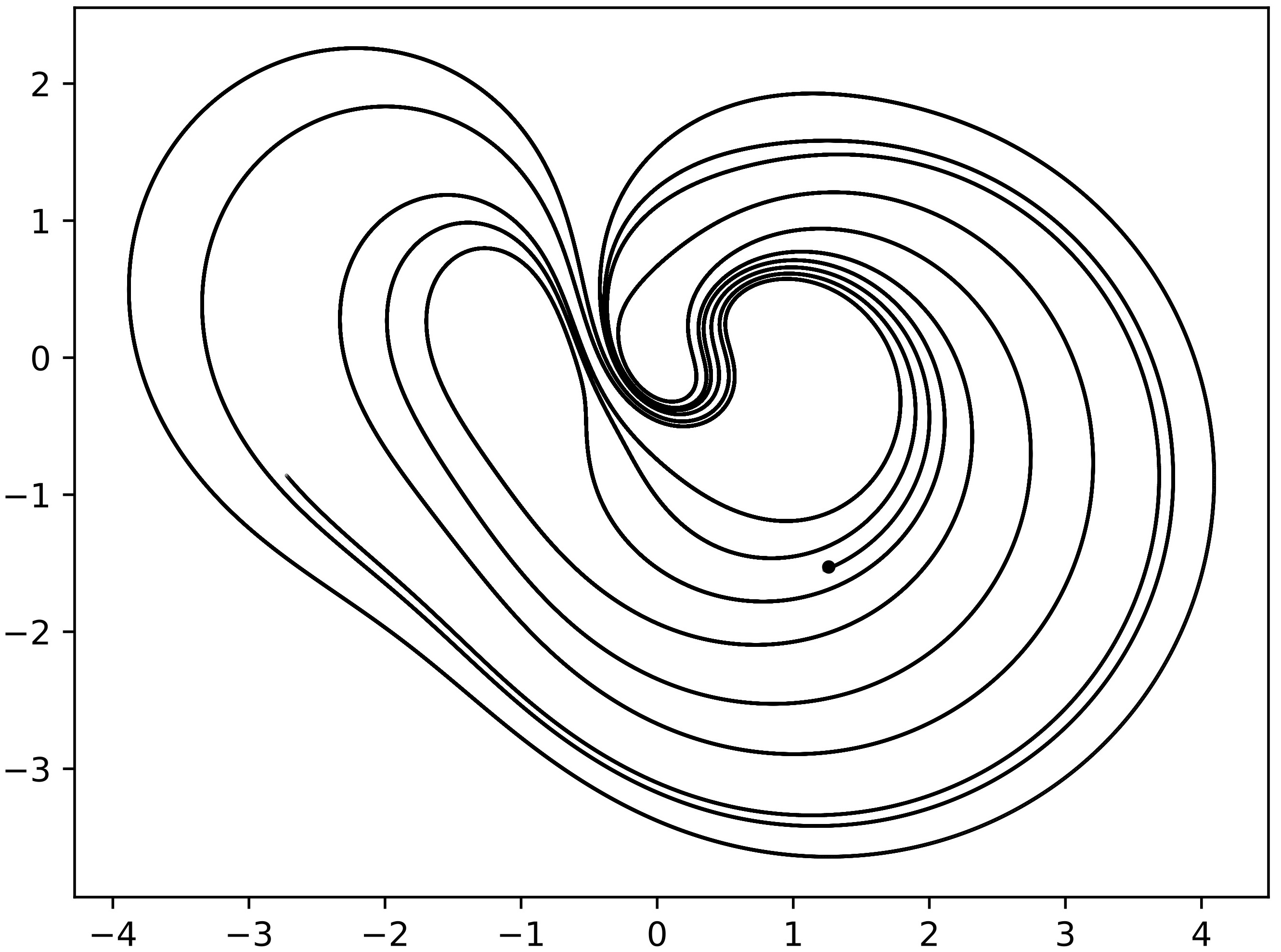}
\caption{A planar Plykin attractor $P$ is generated by the spherical equations of \cite{kun}, followed by a  stereographic projection from the point $\langle 1/\sqrt{2},0,1/\sqrt{2}\rangle$. As with the Ikeda attractor,  approximations of a self-entwined ray in $P$ are obtained by iteration on an arc $I$  which ends at a fixed point of the attractor (in planar coordinates the fixed point is $p\approx \langle 1.26,-1.53\rangle$). This self-entwined ray is remarkable in that it   has bounded curvature.}
\end{figure}

\

The images in Figures 1 and 2 were created in Python via precise numerical calculations.

\section{Preliminaries}

\subsection{Arcs of rays} An \textbf{arc} is a topological  copy of the interval $[0,1]$. If   $X$ is any ray with an associated mapping $f$,  and $[a,b]\subset [0,\infty)$ is a closed and bounded interval, then $f\restriction [a,b]$ is a homeomorphism and $f[a,b]$ is an arc in $X$. By an \textbf{arc of} the ray $X$, we mean an arc of the form $f[a,b]$. It is a   consequence of Sierpiński's theorem \cite[Theorem 5.16]{nad} that if $X$ is not compact then every continuum in $X$ is of this form; see  \cite[Lemma 2.2]{oa}.

\subsection{Definition of self-entwined}Below is an equivalent formulation of the self-entwined property in terms of the ray function.

\begin{up}\label{p2}If $X$ is a ray and $f$ is a continuous one-to-one mapping of $[0,\infty)$ onto $X$, then the following are equivalent. 
\begin{itemize}
\item[\textnormal{(a)}] $X$ is self-entwined;
    \item[\textnormal{(b)}] $X$ has no compact neighborhoods;
    \item[\textnormal{(c)}] each arc of $X$ has empty interior in $X$;
    \item[\textnormal{(d)}] for every $x\in X$ there exists a sequence $r_n\in [0,\infty)$ such that $r_n\to \infty$  and $f(r_n)\to x$. 
\end{itemize}
\end{up}

\begin{proof}(b)$\Leftrightarrow$(c) by the Baire category theorem. (c)$\Leftrightarrow$(d) is obvious. Clearly (a)$\Rightarrow$(c), and ((b)+(c))$\Rightarrow$(a) by \cite[Lemma 2.2]{oa}. \end{proof}

\subsection{Connectedness im-kleinen}A space $X$ is said to be  \textbf{connected im-kleinen} at a point $x\in X$ if every neighborhood of $x$ contains a connected neighborhood of $x$.

The following is due to Fitzpatrick and Lelek \cite{fit}.

\begin{up}\label{p3}Let $X$ be a continuum. If $X$ is Suslinian, then $X$ is connected im-kleinen at each point of a dense subset. \end{up}

\begin{proof}[Idea of proof]If $U$ is a non-empty open subset of $X$ that has only countably many non-degenerate components, then some  component of $U$ must have interior in $X$ by the Baire category theorem. So if $X$ is Suslinian, it is possible to recursively define a decreasing sequence of continua with interiors in $X$, which intersect down to a single point of $X$ that is contained in each interior. \end{proof}

\subsection{Discs and theta curves} A \textbf{$\theta$-curve} is a space made of three arcs $\alpha_1,\alpha_2,\alpha_3$ which have the same endpoints $a$ and $b$ and are otherwise disjoint. 

It is a standard fact of plane topology that one of the three arcs comprising a $\theta$-curve is   separated from $\infty$ by the union of the  other two.  Another fact relevant to the proposition below is  that if $F$ is a closed subset of $\mathbb R^2$ which separates two points $p$ and $q$ (that is, $p$ and $q$ belong to different components of $\mathbb R^2\setminus F$), then some component of $F$ separates $p$ and  $q$. This  is a well-known corollary of Janiszewski's theorem.

\begin{up}\label{p4}Let $\theta=\alpha_1\cup \alpha_2\cup \alpha_3$ be a theta curve of arcs with endpoints $a$ and $b$, such that $\alpha_1\cup \alpha_3$ separates $\alpha_2$ from $\infty$. Let $D$ be a closed disc missing $a$ and $b$. If  $p\in \alpha_1\cap D^\mathrm{o}$ and $q\in \alpha_3\cap D^\mathrm{o}$, then there  exists an arc $\beta\subset \alpha_2\cap D$ which separates $p$ from $q$ in $D$, and is such that $\beta\cap \partial D$ consists only  of the two endpoints of $\beta$.  \end{up}

\begin{proof}[Sketch of proof]Let $\alpha_4$ be an arc that lies in the unbounded component of $\mathbb R^2\setminus (\alpha_1\cup \alpha_3)$, which has endpoints $a$ and $b$, and is such that $\alpha_4\cap D=\varnothing$. The simple closed curve $\alpha_2\cup \alpha_4$ separates $p$ and $q$. It follows that $\alpha_2\cap D$ separates $p$ from $q$ in $D$. Thus some component $\gamma$ of $\alpha_2\cap D^\mathrm{o}$ separates $p$ from $q$ in $D^\mathrm{o}$. Let $\beta=\overline\gamma$ (see Figure 3).\end{proof}

\begin{figure}[h]\includegraphics[scale=0.38]{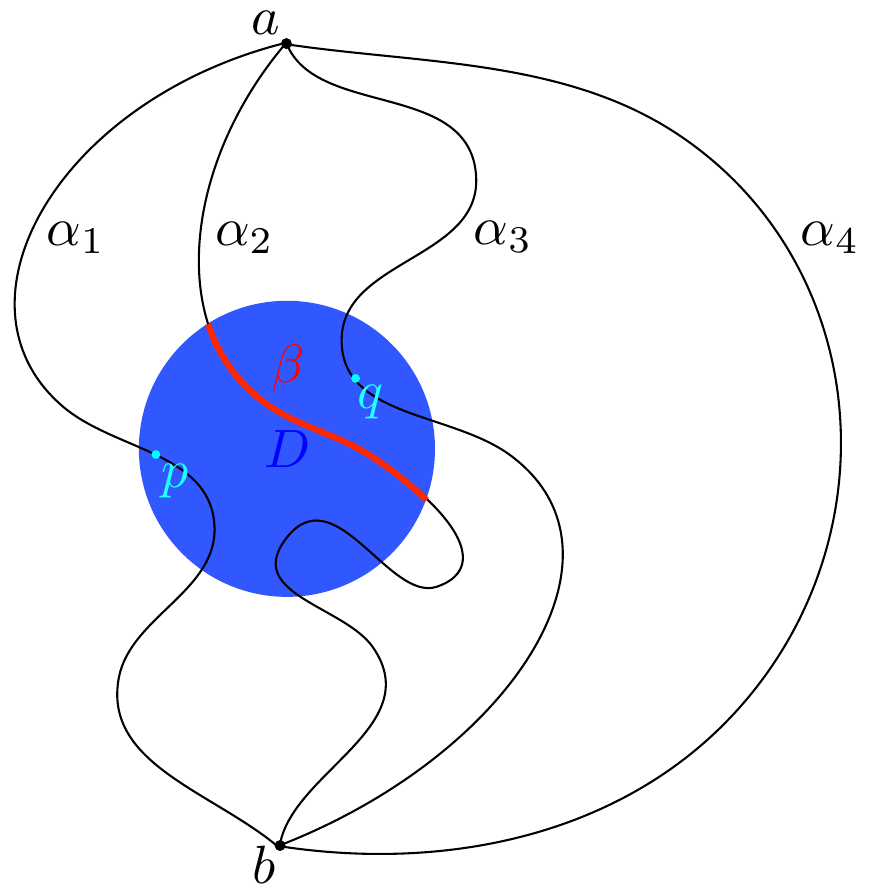}
\caption{Illustration for Proposition \ref{p4}.}
\end{figure}

\subsection{Notation} For any set $A\subset \mathbb R^2$ we let  $\overline A$, $\partial A$, and $A^{\mathrm{o}}$ denote the closure, boundary, and  interior of $A$, respectively, relative to the topology of $\mathbb R^2$. We denote by $2^n$ the set of functions from $\{0,\ldots, n-1\}$ to $\{0,1\}$. Thus an element  $\sigma\in 2^n$ can be viewed as a  binary sequence of length $n$. That is,  $\sigma=\langle \sigma(0),\ldots,\sigma(n-1)\rangle$ where $\sigma(i)\in \{0,1\}$ for each $i<n$.  $$2^{<\omega}=\bigcup _{n=0}^\infty 2^n$$ denotes the set of all finite binary sequences. Finally,  $2^\omega$ is the set of all infinite binary sequences index by $\omega=\{0,1,2,3,\ldots\}$. If $\sigma\in 2^{<\omega}$, then $\sigma^\frown 0$  and $\sigma^\frown 1$ represent the one-element extensions of $\sigma$ obtained by appending $0$ and $1$, respectively. If $\tau\in 2^{ \omega}$ and $n<\omega$, then $\tau\restriction n$ is the sequence made of the first $n$ terms of $\tau$.

\section{Rays that limit to  arcs on both sides}

Let $X$ be a ray in $\mathbb R^2$. Let $f$ be a continuous one-to-one mapping of $[0,\infty)$ onto $X$. For each integer $n\geq 1 $ let $S_n$ be a simple closed curve containing the arc $f[0,n]$. Let $U_n$ and $V_n$ be the components of $\mathbb R^2\setminus S_n$. Given $t>0$ and an integer $n>t$, we say that $X$ \textbf{limits to $f(t)$ from both sides of} $f[0,n]$ if  $f(t)\in \overline{U_n\cap X}\cap \overline{V_n\cap X}$. 


A point $x=f(t)\in X$ is a \textbf{two-sided limit point of }$X$ if there exists an integer $n>t$ such that  $X$ limits to $f(t)$ from both sides of $f[0,n]$.


Let $A,B,C\subset \mathbb R^2$. An arc $\alpha\subset \mathbb R^2$ is \textbf{$ABC$-minimal} if $\alpha\cap A$ and $\alpha\cap C$ are the endpoints of $\alpha$, and $\alpha\cap B\neq\varnothing$. It is  easy to prove that if $A$, $B$, and $C$ are closed and pairwise-disjoint,  and $\alpha$ is  any arc meeting all three sets, then $\alpha$ contains an arc which is $ABC$,  $BAC$, or $ACB$-minimal.

\begin{ul}\label{le5}If $X$ is self-entwined and $\overline X$ is connected im-kleinen at a dense set of points, then the set of two-sided limit points of $X$ is dense in $X$.\end{ul}

\begin{proof}Suppose that $X$ is self-entwined and $\overline X$ is connected im-kleinen at a dense set of points. Let $U$ be any non-empty open subset of $\mathbb R^2$ meeting $X$. Let $D_1,D_2,D_3$ be pairwise-disjoint closed discs in $U$  whose interiors intersect $\overline X$. For each $i\in \{1,2,3\}$ let $K_i\subset X$ be a continuum which has interior in $X$   and is contained  in the interior of $D_i$.   There are 7 pairwise-disjoint   arcs in $X$ which meet every $K_i$. Each one contains an arc that is $K_1K_2K_3$, or $K_2K_1K_3$, or $K_2K_3K_1$-minimal.  At least 3 of the 7 minimal arcs, say $\alpha_1,\alpha_2,\alpha_3$, are of the same type.  By re-indexing the $D's$ and $K's$ we may assume that the three arcs are $K_1K_2K_3$-minimal. Upon identifying $K_1$ and $K_3$ with points $a$ and $b$, we have that $\alpha_1\cup \alpha_2\cup \alpha_3$ is a $\theta$-curve. Let us assume that $\alpha_2$ is separated from $\infty$ by $\alpha_1\cup \alpha_3$. Choose $p\in \alpha_1\cap K_2$ and $q\in \alpha_3\cap K_2$. By Proposition \ref{p4}, there exists an arc $\beta\subset \alpha_2\cap D_2$ which separates $p$ from $q$ in $D_2$, and is such that  $\beta\cap \partial D_2$ consists of only the two endpoints of $\beta$.  We claim that $\beta$ contains a two-sided limit point of $X$.

Let $n$ be an integer such that $\beta\subset f[0,n]$. Since $\beta$ is nowhere dense in $X$, $\beta$ is contained in the union of $\overline{U_n\cap X}$ and $\overline{V_n\cap X}$. If each of these sets intersects $\beta$, then by connectedness of $\beta$ there exists $x\in\overline{U_n\cap X}\cap \overline{V_n\cap X}\cap \beta.$ Then $x$ is a two-sided limit point of $X$ in $U$. Otherwise, one of the two sets misses $\beta$  which we will show leads to a contradiction. The contradiction will be that there exists an arc $\gamma$ that misses $K_2$, and separates $p$ and $q$ in $D_2$.

Without loss of generality, assume that $\overline{V_n\cap X}\cap \beta=\varnothing$. Let $G$ be the component of $\beta$ in $D_2\cap \overline{V_n}$. Since $\beta\cap \partial D_2$ consists of only the two endpoints of $\beta$, we have that $G$ is a closed topological disc and $\beta\subset \partial G$. Further,  $G\setminus \overline{V_n\cap X}$ is a neighborhood of $\beta$ in  the space $G$. Thus there exists an arc $\gamma\subset G\setminus \overline{V_n\cap X}$ which has the same endpoints as $\beta$ and is otherwise contained in $G^\mathrm{o}$. Choosing $\gamma$ close to $\beta$ (closer in Hausdorff distance than $d(p,\beta)$ and $d(q,\beta)$) will ensure that it separates $D_2$ between $p$ and $q$. Moreover, $\gamma\cap \overline X\subset \partial D_2$ by the fact $(G^\mathrm{o} \setminus \overline{V_n\cap X})\cap\overline X\subset (V_n \setminus \overline{V_n\cap X})\cap\overline X=\varnothing.$  Therefore $\gamma$ misses $K_2$, which  contradicts that $K_2$ is connected. \end{proof}

 \begin{ur}If $\overline X$ is indecomposable, the set of two-sided limit points of $X$ may be empty. This is the case when $X$ is the accessible composant (arc component of the endpoint) of the Knaster buckethandle continuum.  \end{ur}

\section{Proof of Theorem \ref{t1}}

Let $X$ be a self-entwined ray in $\mathbb R^2$. Let $f:[0,\infty)\to X$ be a continuous one-to-one surjection. We want to show that $\overline X$ is non-Suslinian. This follows immediately from Proposition \ref{p3} if  $\overline X$ is not connected im-kleinen at a dense set of points. For the remainder of the proof we may therefore assume that: \textit{$\overline X$ is connected im-kleinen at a dense set of points.}  


There exists a proper subcontinuum $K$ of $\overline X$ with interior in $\overline X$, and a disc $D\subset \mathbb R^2$ such that $D$ lies   in the unbounded component of $\mathbb R^2\setminus K$ and $D^\mathrm{o}\cap X\neq\varnothing$. We will show that $\overline X$ contains  an uncountable collection of pairwise-disjoint subcontinua stretching from $K$ to $D$.


By working in the unbounded component of $\mathbb R^2\setminus K$, we may assume that $K$ is non-separating.  Hence  if $\alpha$ is any $K(X\setminus K)K$-minimal arc, then $K\cup \alpha$ has exactly two complementary components. We may also assume that the ray's endpoint $f(0)$ belongs to $K$, because there exists  $f(t)\in K$ and  $f[t,\infty)$ is a self-entwined ray with the same closure as $X$.  Under this assumption, for every  $x\in X\setminus K$ there exists a (unique) $K\{x\}K$-minimal arc in $X$. This arc is found by tracing $x$ to $K$,  backwards and forwards along the ray.  


 For each $x\in X\setminus K$ we may now define:

\begin{itemize}
\item $\alpha(x)$ to  be the unique $K\{x\}K$-minimal arc of $X$, 
\item $U(x)$  to be the bounded component of $\mathbb R^2\setminus (K\cup \alpha(x))$, and  
\item $V(x)$  to be the unbounded component of $\mathbb R^2\setminus (K\cup \alpha(x))$. 

\end{itemize}

\begin{ur}Note that $\alpha(x)\subset \partial U(x)\subset \alpha(x)\cup K$ and likewise for $V(x)$. Additionally,  $\{U(x),\alpha(x),V(x)\}$ partitions the $K(X\setminus K)K$-minimal arcs of $X$, with the possible exception that $\alpha(x)$ may share an endpoint in $K$ with an arc in $U(x)$ or $V(x)$.\end{ur}

\begin{uobs}\label{o6}
 If $x$ is any two-sided limit point of $X$ in $X\setminus K$,  then $x\in \overline {U(x)\cap X}\cap \overline{V(x)\cap X}.$\end{uobs}
 
 \begin{proof}  Suppose that $x=f(t)$ is a two-sided limit point in $X\setminus K$. Let $n>t$ be such that $X$ limits to $x$ from both sides of $f[0,n]$. There exists a closed domain  $G$ such that $x\in G^{\mathrm{o}}$,
 \begin{align*}
 G&\cap K=\varnothing,\\
 G&\cap S_n\subset f(0,n)\text{, and}\\
 G&\cap \alpha(x)=G\cap f[0,n]=\beta,
 \end{align*}
 where $\beta$ is an arc  that contains $x$ and  cuts $G$ into exactly two pieces. This is easy to see if one uses the fact that $f[0,n]$ is ambiently homeomorphic to a straight arc in which $f[0,n]\cap \alpha(x)$ is a neighborhood of $x$.  
 
 Let $G_0$ and $G_1$ be the components of $G\setminus \beta$. Since $G\setminus \beta\subset \mathbb R^2\setminus (K\cup \alpha(x))$ and $U(x)$ and $V(x)$ each intersect $G$, without loss of generality we may assume $G_0\subset U(x)$ and $G_1\subset V(x)$. Likewise, since $G\setminus \beta\subset \mathbb R^2\setminus S_n$ and $U_n$ and $V_n$ each intersect $G$, there exists $i\in \{0,1\}$ such that  $G_i\subset U_n$ and $G_{1-i}\subset V_n$. Then $G\cap U_n=G_i$ and $G\cap V_n=G_{1-i}$.  Now let $W$ be any open set containing $x$. Since $X$ limits to $x$ from both sides of $f[0,n]$, $W\cap G_i \cap X$ and $W\cap G_{1-i} \cap X$ are each non-empty. Therefore $W\cap U(x) \cap X$ and $W\cap V(x) \cap X$ are each non-empty. This shows that $x\in \overline {U(x)\cap X}\cap \overline{V(x)\cap X}$. \end{proof}

We will now define a Cantor tree of compact regions of $\mathbb R^2$, $$\big\{D_\sigma: \sigma\in 2^{<\omega}\big\},$$  and two-sided limit points $x_\sigma\in X $ so that:

\begin{enumerate}
\item $\alpha(x)\subset D_\sigma$ for all $x\in D_{\sigma}\cap X$,

\item $x_{\sigma^\frown i}\in (D\cap D_\sigma)^{\mathrm{o}}$ for each $i\in \{0,1\}$, and

\item $\overline{U(x_{\sigma^\frown 1})}\setminus K\subset U(x_{\sigma^\frown 0})$.
\end{enumerate}
Each $D_\sigma$ will be bounded by $K$ and finitely-many arcs of $X$.

\medskip
 
\textit{First step}: To begin, let $a \in D^\mathrm{o}\cap X$ be a two-sided limit point of $X$ provided by Lemma \ref{le5}. By Observation \ref{o6} and Lemma \ref{le5} there exist  two-sided limit points $b\in D^\mathrm{o}\cap U(a)\cap X$, $ c\in D^\mathrm{o}\cap U(b) \cap X$, and $d\in D^\mathrm{o}\cap U(c) \cap X$. Define $D_\varnothing=\overline{U(a)}\setminus U(d)$ and put $x_0=b$ and $x_1=c$.  

It is easy to see that (1) and (2) hold for $\sigma=\varnothing$. As for (3), $\overline{U(x_{1})}\setminus K\subset  U(x_{1})\cup \alpha (x_1) \subset U(x_{0})$. 

\medskip
 
\textit{Inductive step}: Now suppose that $\sigma\in 2^n$ and $D_\sigma$ and $x_{\sigma^\frown 0}$ and $x_{\sigma^\frown 1}$ satisfy (1) through (3). Define  
\begin{align*} D_{\sigma^\frown 0}&=D_\sigma\cap \overline{V(x_{\sigma^\frown 0})}\\
 D_{\sigma^\frown 1}&=D_\sigma\cap \overline{U(x_{\sigma^\frown 1})}.\end{align*}

 By (2), Observation \ref{o6} and Lemma \ref{le5}, there exists a two-sided limit point $x_{\sigma^\frown 0,0}\in (D\cap D_{\sigma})^\mathrm{o}\cap V(x_{\sigma^\frown 0})\cap X$. And  there exists a two-sided limit point $x_{\sigma^\frown 0,1}\in (D\cap D_{\sigma})^\mathrm{o}\cap V(x_{\sigma^\frown 0})\cap U(x_{\sigma^\frown 0,0})\cap X.$ Points   $x_{\sigma^\frown 1,0}$ and $x_{\sigma^\frown 1,1}$ in $D_{\sigma^\frown 1}$ can be found in the same way, with  $U(x_{\sigma^\frown 1})$ in the place of $V(x_{\sigma^\frown 0})$. See Figures 4 and 5.

 (1) through (3) now hold for all   $\sigma\in 2^{n+1}$.  (1) is verified by combining  the condition on  $D_{\sigma\restriction n}$ with the fact that $\overline{V(x_{\sigma^\frown 0})}\cap X$ and $\overline{U(x_{\sigma^\frown 1})}\cap X$ are  unions of arcs  $\alpha(x)$ (see Remark 2). (2) is obvious, and (3) holds for the reason given in the base case.

 \newpage

\begin{figure}[h]
\includegraphics[scale=0.2]{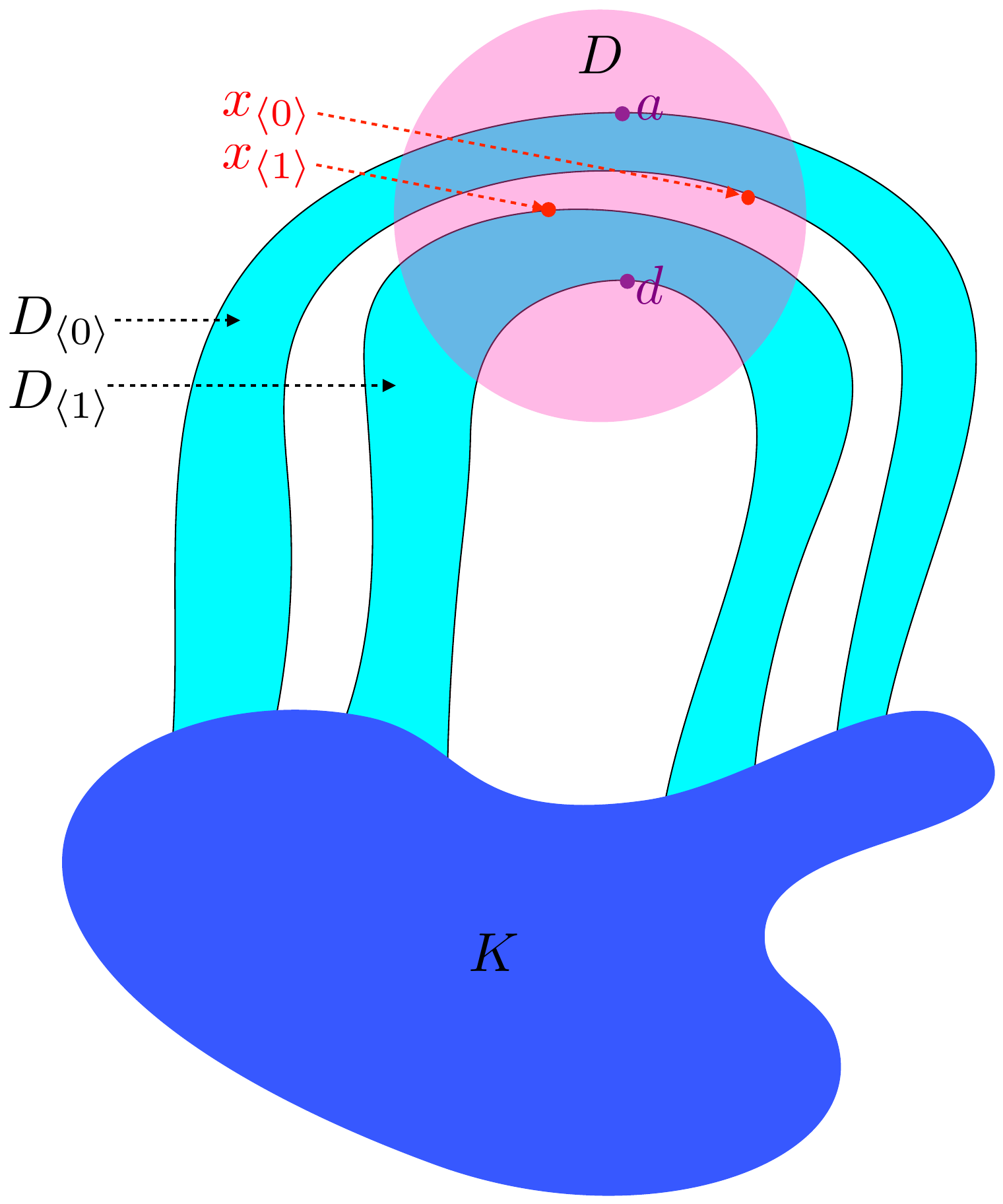}
\caption{Construction of $D_{\langle 0\rangle}$ and $D_{\langle 1\rangle}$.}
\end{figure}

\begin{figure}[h]
\includegraphics[scale=0.2]{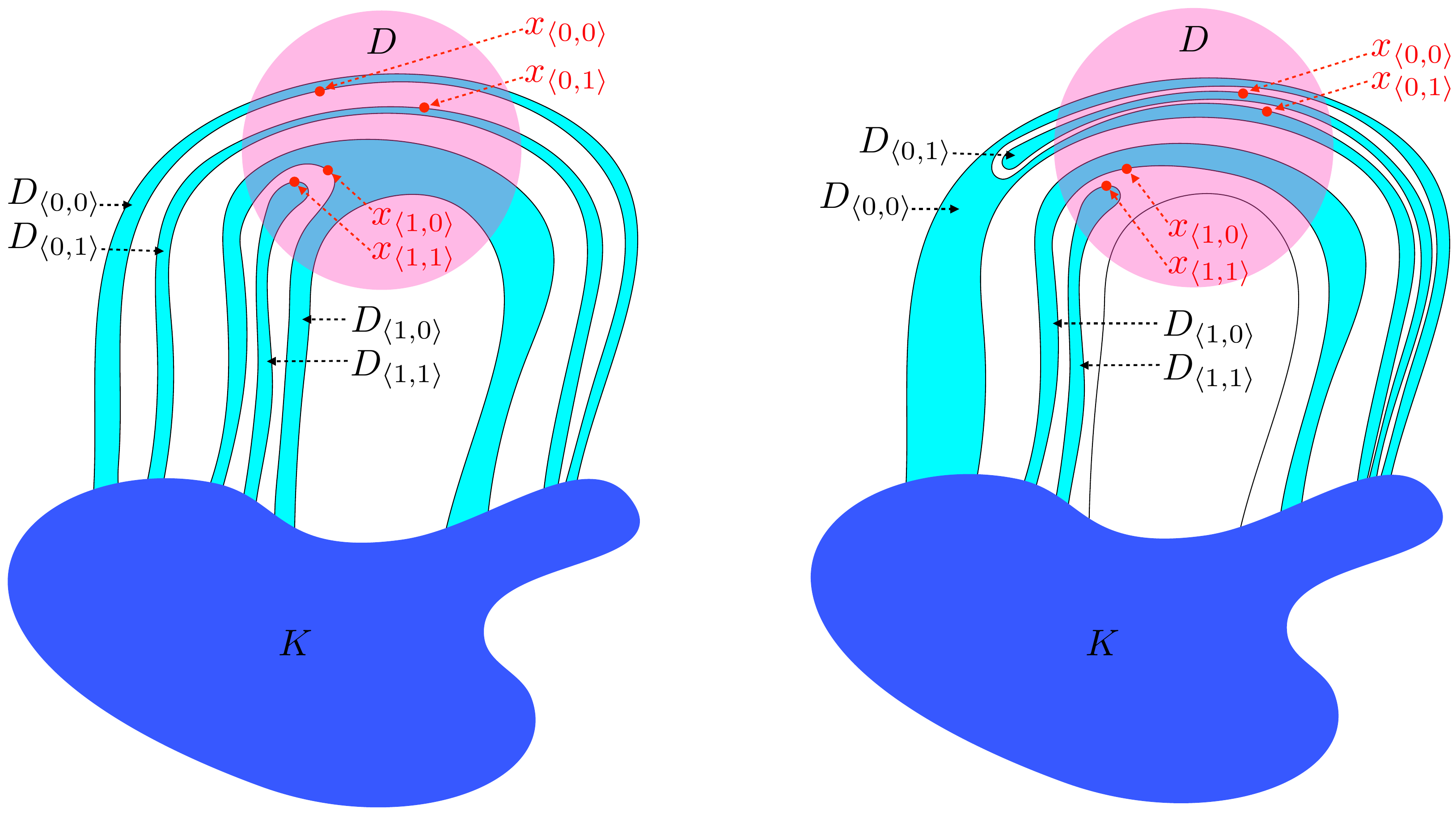}
\caption{Two possible arrangements of $D_\sigma$  ($\sigma\in 2^2$) within $D_{\langle 0\rangle}$ and $D_{\langle 1\rangle}$.}
\end{figure}

For all $\sigma\in 2^{<\omega}$ we have:

\begin{itemize}\renewcommand{\labelitemi}{\Tiny$\blacksquare$}
\item $K\cup (D_{\sigma}\cap X)$ is a connected set that meets both $K$ and $D$ by (1) and (2),
\item $D_{\sigma^\frown 0},D_{\sigma^\frown 1}\subset D_{\sigma}$ by definition, and
\item $D_{\sigma^\frown 0}\cap D_{\sigma^\frown 1}\setminus K=\varnothing$ by (3). 
\end{itemize}

For every  $\tau \in 2^\omega$ define $$K_\tau=\bigcap_{n=0}^\infty\overline{K\cup (D_{\tau\restriction n}\cap X)}.$$ Then $K_\tau$ is the intersection of a decreasing sequence of continua which contain $K$ and intersect $D$. Therefore $K_\tau$ is  a continuum that contains $K$ and intersects $D$ \cite[p.6]{nad}.

Let $W$ be an open set containing $K$ such that $\overline W\cap D=\varnothing$. For each  $\tau\in 2^\omega$ let $M_\tau$ be the component of some  point of $K_\tau\cap D$ in $K_\tau\setminus W$. Note that different $\tau's$ produce disjoint $M_{\tau}$'s. And by the boundary bumping principle  each   $M_\tau$ is a non-degenerate subcontinuum of $\overline X$ \cite[Theorem 5.4]{nad}. We conclude that $\{M_\tau:\tau\in 2^\omega\}$ is an uncountable collection of pairwise-disjoint, non-degenerate subcontinua of $\overline X$. Therefore $\overline X$ is non-Suslinian. 

This completes the proof of Theorem \ref{t1}. \hfill $\blacksquare$

\end{document}